\documentclass[12pt,fleqn]{article}

\usepackage{amsmath,amssymb,amsthm}
\usepackage{geometry} 
\usepackage[pdfpagemode=UseNone,pdfstartview=FitH,hypertex]{hyperref}

\newcommand{\bdry}[1]{\partial #1}
\newcommand{\dint}{\ds{\int}}
\newcommand{\ds}[1]{\displaystyle #1}
\newcommand{\eps}{\varepsilon}
\newcommand{\half}{\frac{1}{2}}
\newcommand{\norm}[2][]{\left\|#2\right\|_{#1}}
\renewcommand{\O}{\text{O}}
\renewcommand{\o}{\text{o}}
\newcommand{\PS}[1]{$(\text{PS})_{#1}$}
\newcommand{\pnorm}[2][]{\if #1'' \left|#2\right|_p \else \left|#2\right|_{#1} \fi}
\newcommand{\QED}{\mbox{\qedhere}}
\newcommand{\R}{\mathbb R}
\newcommand{\seq}[1]{\left(#1\right)}
\newcommand{\set}[1]{\left\{#1\right\}}
\newcommand{\wto}{\rightharpoonup}

\newenvironment{enumroman}{\begin{enumerate}
\renewcommand{\theenumi}{$(\roman{enumi})$}
\renewcommand{\labelenumi}{$(\roman{enumi})$}}{\end{enumerate}}

\newenvironment{properties}[1]{\begin{enumerate}
\renewcommand{\theenumi}{$(#1_\arabic{enumi})$}
\renewcommand{\labelenumi}{$(#1_\arabic{enumi})$}}{\end{enumerate}}

\newtheorem{corollary}{Corollary}[section]
\newtheorem{lemma}[corollary]{Lemma}
\newtheorem{proposition}[corollary]{Proposition}
\newtheorem{theorem}[corollary]{Theorem}

\theoremstyle{definition}
\newtheorem{definition}[corollary]{Definition}

\theoremstyle{remark}
\newtheorem{remark}[corollary]{Remark}

\numberwithin{equation}{section}

\title{\bf On the compactness threshold in the critical Kirchhoff equation}
\author{\bf Erisa Hasani and Kanishka Perera\\
Department of Mathematical Sciences\\
Florida Institute of Technology\\
Melbourne, FL 32901, USA\\
\em ehasani2016@my.fit.edu \& kperera@fit.edu}
\date{}

\begin{document}

\maketitle

\begin{abstract}
We study a class of critical Kirchhoff problems with a general nonlocal term. The main difficulty here is the absence of a closed-form formula for the compactness threshold. First we obtain a variational characterization of this threshold level. Then we prove a series of existence and multiplicity results based on this variational characterization.
\end{abstract}

\section{Introduction}

The purpose of this paper is to study the critical Kirchhoff problem
\begin{equation} \label{1.1}
\left\{\begin{aligned}
- h\bigg(\int_\Omega |\nabla u|^2\, dx\bigg)\, \Delta u & = f(x,u) + |u|^{2^\ast-2}\, u && \text{in } \Omega\\[5pt]
u & = 0 && \text{on } \bdry{\Omega},
\end{aligned}\right.
\end{equation}
where $\Omega$ is a bounded domain in $\R^N,\, N \ge 3$, $h : [0,\infty) \to [0,\infty)$ is a continuous and nondecreasing function, $2^\ast = 2N/(N - 2)$ is the critical Sobolev exponent, and $f$ is a Carath\'{e}odory function on $\Omega \times \R$ satisfying the subcritical growth condition
\begin{equation} \label{1.2}
|f(x,t)| \le c_1 |t|^{p-1} + c_2 \quad \text{for a.a.\! } x \in \Omega \text{ and all } t \in \R
\end{equation}
for some constants $c_1, c_2 > 0$ and $1 < p < 2^\ast$. The class of nonlocal terms considered here includes sums of powers
\[
h(t) = \sum_{i=1}^n a_i\, t^{\gamma_i-1}, \quad t \ge 0,
\]
where $a_1, \dots, a_n > 0$ and $1 \le \gamma_1 < \cdots < \gamma_n < + \infty$. A model case is
\[
h(t) = a + bt^{\gamma-1},
\]
where $a, b \ge 0$ with $a + b > 0$ and $1 < \gamma < + \infty$. The classical case $h(t) = a + bt$ corresponds to $\gamma = 2$.

As is usually the case with problems of critical growth, problem \eqref{1.1} lacks compactness. The standard approach to such problems is to determine a threshold level below which there is compactness and construct minimax critical levels below this threshold. This approach has been used in the classical case
\[
h(t) = a + bt, \quad a, b \ge 0, \quad a + b > 0
\]
in dimensions $N = 3 \text{ and } 4$ to obtain nontrivial solutions in the recent literature (see, e.g., Huang et al.\! \cite{MR3771258}, Liao et al.\! \cite{MR3861499}, Naimen \cite{MR3210026,MR3278854}, Xie et al.\! \cite{MR3060908}, Yao and Mu \cite{MR3529380}, Zhang and Liu \cite{MR3626222}, and the references therein). However, in the general case considered in the present paper such a threshold level cannot be found in closed form. Our first contribution here is a variational characterization of this threshold level (see Theorem \ref{Theorem 2.3}). Then we give a series of existence and multiplicity results based on this variational characterization (see Section \ref{Section 3}). This requires novel arguments due to the absence of a closed-form compactness threshold.

We will state and prove our compactness, existence, and multiplicity results for problem \eqref{1.1} with a general nonlocal term $h$ in the next two sections. To illustrate our results while keeping the presentation simple, we state them here for the model problem
\begin{equation} \label{1.3}
\left\{\begin{aligned}
- \left[a + b \left(\int_\Omega |\nabla u|^2\, dx\right)^{\gamma-1}\right] \Delta u & = \lambda u + |u|^{2^\ast-2}\, u && \text{in } \Omega\\[5pt]
u & = 0 && \text{on } \bdry{\Omega},
\end{aligned}\right.
\end{equation}
where $1 < \gamma < + \infty$, $a \ge 0$, $b > 0$, and $\lambda > 0$.

Weak solutions of this problem coincide with critical points of the functional
\[
J(u) = \frac{a}{2} \int_\Omega |\nabla u|^2\, dx + \frac{b}{2 \gamma} \left(\int_\Omega |\nabla u|^2\, dx\right)^\gamma - \frac{\lambda}{2} \int_\Omega u^2\, dx - \frac{1}{2^\ast} \int_\Omega |u|^{2^\ast} dx, \quad u \in H^1_0(\Omega).
\]
Recall that $J$ satisfies the Palais-Smale compactness condition at the level $c \in \R$, or the \PS{c} condition for short, if every sequence $\seq{u_j}$ in $H^1_0(\Omega)$ such that $J(u_j) \to c$ and $J'(u_j) \to 0$ has a strongly convergent subsequence. Let $S$ be the best Sobolev constant (see \eqref{1.7}) and let $\lambda_1 > 0$ be the first Dirichlet eigenvalue of the Laplacian in $\Omega$. We have the following compactness results for the cases $\gamma < 2^\ast/2$, $\gamma = 2^\ast/2$, and $\gamma > 2^\ast/2$ (see Corollary \ref{Corollary 1.6}, Corollary \ref{Corollary 1.7}, and Corollary \ref{Corollary 2.11}).

\begin{theorem} \label{Theorem 1.1}
Let $1 < \gamma < 2^\ast/2$, $a, b > 0$, and $0 < \lambda \le a \lambda_1$. Let $t_0$ be the unique positive solution of the equation $a + bt^{\gamma-1} = S^{-2^\ast\!/2}\, t^{2^\ast\!/2-1}$ and set
\[
c^\ast = \frac{1}{N}\, at_0 + \left(\frac{1}{2 \gamma} - \frac{1}{2^\ast}\right) bt_0^\gamma.
\]
Then $J$ satisfies the {\em \PS{c}} condition for all $c < c^\ast$.
\end{theorem}

\begin{theorem}
Let $\gamma = 2^\ast/2$.
\begin{enumroman}
\item Let $a > 0$, $0 < b < S^{-2^\ast\!/2}$, and $0 < \lambda < a \lambda_1$. Set
    \[
    c^\ast = \frac{1}{N} \left(\frac{a^{2^\ast\!/2}}{S^{-2^\ast\!/2} - b}\right)^{2/(2^\ast-2)}.
    \]
    Then $J$ satisfies the {\em \PS{c}} condition for all $c < c^\ast$.
\item If $a \ge 0$ and $b > S^{-2^\ast\!/2}$, then $J$ satisfies the {\em \PS{c}} condition for all $c \in \R$ for any $\lambda > 0$.
\end{enumroman}
\end{theorem}

\begin{theorem} \label{Theorem 1.3}
If $\gamma > 2^\ast/2$ and
\[
a^{\gamma-2^\ast\!/2}\, b^{2^\ast\!/2-1} > \frac{(\gamma - 2^\ast/2)^{\gamma-2^\ast\!/2}\, (2^\ast/2 - 1)^{2^\ast\!/2-1}}{(\gamma - 1)^{\gamma-1}}\, S^{-(2^\ast\!/2)(\gamma-1)},
\]
then $J$ satisfies the {\em \PS{c}} condition for all $c \in \R$ for any $\lambda > 0$.
\end{theorem}

Theorems \ref{Theorem 1.1}--\ref{Theorem 1.3} have the following corollary for the classical case $\gamma = 2$, where $c^\ast = + \infty$ means that $J$ satisfies the \PS{c} condition for all $c \in \R$.

\begin{corollary} \label{Corollary 1.4}
Let $\gamma = 2$.
\begin{enumroman}
\item \label{Corollary 1.4.i} If $N = 3$, $a, b > 0$, and $0 < \lambda \le a \lambda_1$, then $c^\ast = \dfrac{1}{4}\, abS^3 + \dfrac{1}{24}\, b^3 S^6 + \dfrac{1}{24} \left(4aS + b^2 S^4\right)^{3/2}$.
\item If $N = 4$, $a > 0$, $0 < b < S^{-2}$, and $0 < \lambda < a \lambda_1$, then $c^\ast = \dfrac{a^2}{4\, (S^{-2} - b)}$.
\item \label{Corollary 1.4.iii} If $N = 4$, $a \ge 0$, and $b > S^{-2}$, then $c^\ast = + \infty$ for any $\lambda > 0$.
\item If $N \ge 5$ and $a^{N-4}\, b^2 > \dfrac{4\, (N - 4)^{N-4}}{(N - 2)^{N-2}}\, S^{-N}$, then $c^\ast = + \infty$ for any $\lambda > 0$.
\end{enumroman}
\end{corollary}

\begin{remark}
The threshold levels in Corollary \ref{Corollary 1.4} \ref{Corollary 1.4.i}--\ref{Corollary 1.4.iii} were also obtained using different arguments in Naimen \cite{MR3278854}, Naimen \cite{MR3210026}, and Liao et al.\! \cite{MR3861499}, respectively.
\end{remark}

We have the following existence and multiplicity results for problem \eqref{1.3} (see Corollary \ref{Corollary 3.6}, Theorem \ref{Theorem 3.10}, and Theorem \ref{Theorem 3.11}).

\begin{theorem} \label{Theorem 1.6}
If $1 < \gamma < 2^\ast/2$, $a, b > 0$, and $N \ge 4$, then problem \eqref{1.3} has a nontrivial solution for $0 < \lambda < a \lambda_1$.
\end{theorem}

\begin{theorem}
Let $\gamma = 2^\ast/2$.
\begin{enumroman}
\item If $a > 0$, $0 < b < S^{-2^\ast\!/2}$, and $N \ge 4$, then problem \eqref{1.3} has a nontrivial solution for $0 < \lambda < a \lambda_1$.
\item If $a = 0$ and $b > S^{-2^\ast\!/2}$, then problem \eqref{1.3} has a nontrivial solution for all $\lambda > 0$.
\item If $a > 0$ and $b > S^{-2^\ast\!/2}$, then problem \eqref{1.3} has two nontrivial solutions for $\lambda > a \lambda_1$.
\end{enumroman}
\end{theorem}

\begin{theorem} \label{Theorem 1.8}
If $\gamma > 2^\ast/2$ and
\[
a^{\gamma-2^\ast\!/2}\, b^{2^\ast\!/2-1} > \frac{(\gamma - 2^\ast/2)^{\gamma-2^\ast\!/2}\, (2^\ast/2 - 1)^{2^\ast\!/2-1}}{(\gamma - 1)^{\gamma-1}}\, S^{-(2^\ast\!/2)(\gamma-1)},
\]
then problem \eqref{1.3} has two nontrivial solutions for $\lambda \ge a \lambda_1$.
\end{theorem}

Theorems \ref{Theorem 1.6}--\ref{Theorem 1.8} have the following corollaries for the classical case $\gamma = 2$.

\begin{corollary} \label{Corollary 1.10}
Let $\gamma = 2$ and $N = 4$.
\begin{enumroman}
\item \label{Corollary 1.10.i} If $a > 0$ and $0 < b < S^{-2}$, then problem \eqref{1.3} has a nontrivial solution for $0 < \lambda < a \lambda_1$.
\item If $a = 0$ and $b > S^{-2}$, then problem \eqref{1.3} has a nontrivial solution for all $\lambda > 0$.
\item If $a > 0$ and $b > S^{-2}$, then problem \eqref{1.3} has two nontrivial solutions for $\lambda > a \lambda_1$.
\end{enumroman}
\end{corollary}

\begin{corollary} \label{Corollary 1.11}
If $\gamma = 2$, $N \ge 5$, and
\[
a^{N-4}\, b^2 > \dfrac{4\, (N - 4)^{N-4}}{(N - 2)^{N-2}}\, S^{-N},
\]
then problem \eqref{1.3} has two nontrivial solutions for $\lambda \ge a \lambda_1$.
\end{corollary}

\begin{remark}
The result in Corollary \ref{Corollary 1.10} \ref{Corollary 1.10.i} was also obtained in Naimen \cite{MR3210026} using a different method. Liao et al.\! \cite{MR3861499} obtained one nontrivial solution when $a \ge 0$, $b > S^{-2}$, and $\lambda > a \lambda_1$. See also Perera and Zhang \cite{MR2193850} for a related result in the subcritical case in dimensions $N \le 3$.
\end{remark}

\begin{remark}
Corollary \ref{Corollary 1.11} complements the results in Naimen and Shibata \cite{MR3987393}, where two positive solutions were obtained when $a = 1$, $b > 0$ is sufficiently small, and $0 < \lambda < \lambda_1$.
\end{remark}

In the borderline case where $\gamma = 2^\ast/2$ and $b = S^{-2^\ast\!/2}$, lower-order terms come into play. Consider the problem
\begin{equation} \label{1.4}
\left\{\begin{aligned}
- \left[a + S^{-2^\ast\!/2} \left(\int_\Omega |\nabla u|^2\, dx\right)^{2^\ast\!/2-1} + \eta \left(\int_\Omega |\nabla u|^2\, dx\right)^{\sigma-1}\right] & = \lambda u + |u|^{2^\ast-2}\, u && \text{in } \Omega\\[5pt]
u & = 0 && \text{on } \bdry{\Omega},
\end{aligned}\right.
\end{equation}
where $a \ge 0$, $\eta > 0$, $1 < \sigma < 2^\ast/2$, and $\lambda > 0$. We have the following existence and multiplicity result (see Theorem \ref{Theorem 3.13}).

\begin{theorem} \label{Theorem 1.13}
Let $\eta > 0$ and $1 < \sigma < 2^\ast/2$.
\begin{enumroman}
\item If $a = 0$, then problem \eqref{1.4} has a nontrivial solution for all $\lambda > 0$.
\item If $a > 0$, then problem \eqref{1.4} has two nontrivial solutions for $\lambda > a \lambda_1$.
\end{enumroman}
\end{theorem}

\section{Compactness threshold}

A weak solution of problem \eqref{1.1} is a function $u$ that belongs to the Sobolev space $H^1_0(\Omega)$ and satisfies
\[
h\bigg(\int_\Omega |\nabla u|^2\, dx\bigg) \int_\Omega \nabla u \cdot \nabla v\, dx = \int_\Omega f(x,u)\, v\, dx + \int_\Omega |u|^{2^\ast-2}\, uv\, dx \quad \forall v \in H^1_0(\Omega).
\]
Weak solutions coincide with critical points of the $C^1$-functional
\[
J(u) = \half\, H\bigg(\int_\Omega |\nabla u|^2\, dx\bigg) - \int_\Omega F(x,u)\, dx - \frac{1}{2^\ast} \int_\Omega |u|^{2^\ast} dx, \quad u \in H^1_0(\Omega),
\]
where $F(x,t) = \int_0^t f(x,s)\, ds$ is the primitive of $f$.

\begin{definition}
The functional $J$ satisfies the Palais-Smale compactness condition at the level $c \in \R$, or the \PS{c} condition for short, if every sequence $\seq{u_j}$ in $H^1_0(\Omega)$ such that
\[
J(u_j) \to c \qquad J'(u_j) \to 0,
\]
called a \PS{c} sequence, has a strongly convergent subsequence.
\end{definition}

Let
\begin{equation} \label{1.7}
S = \inf_{u \in H^1_0(\Omega) \setminus \set{0}}\, \frac{\dint_\Omega |\nabla u|^2\, dx}{\left(\dint_\Omega |u|^{2^\ast} dx\right)^{2/2^\ast}}
\end{equation}
be the best Sobolev constant. The set
\[
I = \set{t > 0 : h(t) \le S^{-2^\ast\!/2}\, t^{2^\ast\!/2-1}}
\]
will play an important role in our compactness results. We begin with a simple but useful proposition.

\begin{proposition} \label{Proposition 2.2}
If $\seq{u_j}$ is a sequence in $H^1_0(\Omega)$ such that
\[
J'(u_j) \to 0, \qquad u_j \wto u, \qquad \norm{u_j - u}^2 \to t,
\]
then either $t = 0$ or $t \in I$. In particular, if $I = \emptyset$, then every bounded sequence $\seq{u_j}$ in $H^1_0(\Omega)$ such that $J'(u_j) \to 0$ has a strongly convergent subsequence.
\end{proposition}

\begin{proof}
Since $J'(u_j) \to 0$,
\begin{equation} \label{2.19}
h\bigg(\int_\Omega |\nabla u_j|^2\, dx\bigg) \int_\Omega \nabla u_j \cdot \nabla v\, dx - \int_\Omega f(x,u_j)\, v\, dx - \int_\Omega |u_j|^{2^\ast-2}\, u_j\, v\, dx = \o(\norm{v})
\end{equation}
for all $v \in H^1_0(\Omega)$, and since $u_j \wto u$ and $\norm{u_j - u}^2 \to t$,
\[
\int_\Omega |\nabla u_j|^2\, dx \to \int_\Omega |\nabla u|^2\, dx + t =: s.
\]
Passing to a renamed subsequence, we may assume that $u_j \to u$ strongly in $L^p(\Omega)$ and a.e.\! in $\Omega$. So taking $v = u_j$ in \eqref{2.19} gives
\begin{equation} \label{2.4}
h(s) \left(\int_\Omega |\nabla u|^2\, dx + t\right) - \int_\Omega uf(x,u)\, dx - \int_\Omega |u_j|^{2^\ast} dx = \o(1),
\end{equation}
while taking $v = u$ and passing to the limit gives
\begin{equation} \label{2.5}
h(s) \int_\Omega |\nabla u|^2\, dx - \int_\Omega uf(x,u)\, dx - \int_\Omega |u|^{2^\ast} dx = 0.
\end{equation}
Since
\[
\int_\Omega |u_j|^{2^\ast} dx - \int_\Omega |u|^{2^\ast} dx = \int_\Omega |u_j - u|^{2^\ast} dx + \o(1)
\]
by the Br{\'e}zis-Lieb lemma (see \cite{MR699419}), subtracting \eqref{2.5} from \eqref{2.4} and using \eqref{1.7} gives
\[
th(s) = \int_\Omega |u_j - u|^{2^\ast} dx + \o(1) \le S^{-2^\ast\!/2} \left(\int_\Omega |\nabla (u_j - u)|^2\, dx\right)^{2^\ast\!/2} + \o(1).
\]
If $t > 0$, then passing to the limit and noting that $h(s) \ge h(t)$ since $s \ge t$ and $h$ is nondecreasing gives $h(t) \le S^{-2^\ast\!/2}\, t^{2^\ast\!/2-1}$, so $t \in I$.
\end{proof}

First we consider the case where $I$ is nonempty. Let
\[
H(t) = \int_0^t h(s)\, ds, \quad t \ge 0
\]
be the primitive of $h$, and set
\[
K(t) = \half\, H(t) - \frac{1}{2^\ast}\, th(t), \quad t \ge 0.
\]
For $1 \le \gamma \le 2^\ast/2$, let
\begin{equation} \label{1.6}
\lambda_1(\gamma) = \inf_{u \in H^1_0(\Omega) \setminus \set{0}}\, \frac{\left(\dint_\Omega |\nabla u|^2\, dx\right)^\gamma}{\dint_\Omega |u|^{2\gamma}\, dx}
\end{equation}
be the first eigenvalue of the nonlinear eigenvalue problem
\[
\left\{\begin{aligned}
- \bigg(\int_\Omega |\nabla u|^2\, dx\bigg)^{\gamma-1}\, \Delta u & = \lambda\, |u|^{2\gamma-2}\, u && \text{in } \Omega\\[5pt]
u & = 0 && \text{on } \bdry{\Omega},
\end{aligned}\right.
\]
which is positive by the Sobolev embedding theorem. We note that $\lambda_1(1) = \lambda_1$, the first Dirichlet eigenvalue of the Laplacian in $\Omega$, and $\lambda_1(2^\ast/2) = S^{2^\ast\!/2}$. We assume that
\begin{properties}{A}
\item \label{A1} for some constants $\alpha_1, \dots, \alpha_n > 0$, $1 \le \gamma_1 < \cdots < \gamma_n < 2^\ast/2$, and $\mu_1 \le \lambda_1(\gamma_1), \dots, \linebreak \mu_n \le \lambda_1(\gamma_n)$ with at least one of the inequalities strict,
    \[
    K(t) \ge \sum_{i=1}^n \alpha_i\, t^{\gamma_i} \quad \forall t \ge 0
    \]
    and
    \[
    F(x,t) - \frac{1}{2^\ast}\, tf(x,t) \le \sum_{i=1}^n \mu_i\, \alpha_i\, |t|^{2\gamma_i} \quad \text{for a.a.\! } x \in \Omega \text{ and all } t \in \R;
    \]
\item \label{A2} $K$ is superadditive, i.e.,
    \[
    K(t_1 + t_2) \ge K(t_1) + K(t_2) \quad \forall t_1, t_2 \ge 0;
    \]
\item \label{A3} $h(t)/t^{2^\ast\!/2-1}$ is strictly decreasing for $t > 0$ and
    \[
    0 \le b := \lim_{t \to + \infty}\, \frac{h(t)}{t^{2^\ast\!/2-1}} < S^{-2^\ast\!/2} < \lim_{t \to 0}\, \frac{h(t)}{t^{2^\ast\!/2-1}} \le + \infty.
    \]
\end{properties}
We note that $K$ is nonnegative by \ref{A1} and hence nondecreasing by \ref{A2}. We have the following theorem.

\begin{theorem} \label{Theorem 2.3}
Assume that $I \ne \emptyset$ and \eqref{1.2}, \ref{A1}, and \ref{A2} hold. Set
\begin{equation} \label{1.8}
c^\ast = \inf_{t \in I}\, K(t).
\end{equation}
Then $J$ satisfies the {\em \PS{c}} condition for all $c < c^\ast$. If, in addition, \ref{A3} holds, then the equation
\begin{equation} \label{1.9}
h(t) = S^{-2^\ast\!/2}\, t^{2^\ast\!/2-1}
\end{equation}
has a unique positive solution $t_0$ and
\begin{equation} \label{2.17}
c^\ast = K(t_0),
\end{equation}
in particular, $c^\ast > 0$.
\end{theorem}

\begin{proof}
First we note that for all $u \in H^1_0(\Omega)$,
\begin{equation} \label{2.18}
K\bigg(\int_\Omega |\nabla u|^2\, dx\bigg) - \int_\Omega \left[F(x,u) - \frac{1}{2^\ast}\, uf(x,u)\right] dx \ge \sum_{i=1}^n \alpha_i \left(1 - \frac{\mu_i}{\lambda_1(\gamma_i)}\right) \left(\int_\Omega |\nabla u|^2\, dx\right)^{\gamma_i}
\end{equation}
by \ref{A1} and \eqref{1.6}.

Let $c < c^\ast$ and let $\seq{u_j}$ be a \PS{c} sequence. Then
\begin{equation} \label{2.1}
\half\, H\bigg(\int_\Omega |\nabla u_j|^2\, dx\bigg) - \int_\Omega F(x,u_j)\, dx - \frac{1}{2^\ast} \int_\Omega |u_j|^{2^\ast} dx = c + \o(1)
\end{equation}
and
\begin{equation} \label{2.21}
h\bigg(\int_\Omega |\nabla u_j|^2\, dx\bigg) \int_\Omega \nabla u_j \cdot \nabla v\, dx - \int_\Omega f(x,u_j)\, v\, dx - \int_\Omega |u_j|^{2^\ast-2}\, u_j\, v\, dx = \o(\norm{v})
\end{equation}
for all $v \in H^1_0(\Omega)$. Taking $v = u_j$ in \eqref{2.21}, dividing by $2^\ast$, and subtracting from \eqref{2.1} gives
\[
K\bigg(\int_\Omega |\nabla u_j|^2\, dx\bigg) - \int_\Omega \left[F(x,u_j) - \frac{1}{2^\ast}\, u_j\, f(x,u_j)\right] dx = c + \o(1 + \norm{u_j}),
\]
which together with \eqref{2.18} implies that $\seq{u_j}$ is bounded in $H^1_0(\Omega)$. So a renamed subsequence of $\seq{u_j}$ converges to some $u$ weakly in $H^1_0(\Omega)$, strongly in $L^p(\Omega)$, and a.e.\! in $\Omega$. For a further subsequence, $\norm{u_j - u}^2$ converges to some $t \ge 0$. We will show that $t = 0$.

Suppose $t \ne 0$. Then $t \in I$ by Proposition \ref{Proposition 2.2} and hence
\begin{equation} \label{2.6}
K(t) \ge c^\ast
\end{equation}
by \eqref{1.8}. As in the proof of Proposition \ref{Proposition 2.2},
\[
\int_\Omega |\nabla u_j|^2\, dx \to \int_\Omega |\nabla u|^2\, dx + t =: s
\]
and
\begin{equation} \label{2.20}
sh(s) - \int_\Omega uf(x,u)\, dx - \int_\Omega |u_j|^{2^\ast} dx = \o(1).
\end{equation}
Moreover, passing to the limit in \eqref{2.1} gives
\[
\half\, H(s) - \int_\Omega F(x,u)\, dx - \frac{1}{2^\ast} \int_\Omega |u_j|^{2^\ast} dx = c + \o(1),
\]
and combining this with \eqref{2.20} gives
\[
c = K(s) - \int_\Omega \left[F(x,u) - \frac{1}{2^\ast}\, uf(x,u)\right] dx.
\]
Since
\[
K(s) \ge K(t) + K\bigg(\int_\Omega |\nabla u|^2\, dx\bigg)
\]
by \ref{A2} and
\[
K\bigg(\int_\Omega |\nabla u|^2\, dx\bigg) - \int_\Omega \left[F(x,u) - \frac{1}{2^\ast}\, uf(x,u)\right] dx \ge 0
\]
by \eqref{2.18}, then
\[
c \ge K(t).
\]
This together with \eqref{2.6} gives $c \ge c^\ast$, contrary to assumption. So $t = 0$.

If \ref{A3} holds, then the equation $h(t)/t^{2^\ast\!/2-1} = S^{-2^\ast\!/2}$ has a unique positive solution $t_0$ and $I = [t_0,\infty)$. Since $K$ is nondecreasing, then $c^\ast = K(t_0)$, in particular, \ref{A1} implies that $c^\ast > 0$.
\end{proof}

Next we consider the case where $I$ is empty. We assume that
\begin{properties}{A}
\setcounter{enumi}{3}
\item \label{A4} $h$ satisfies one of the following conditions:
    \begin{enumerate}
    \item[$(i)$] for some constants $\eta > 0$ and $p/2 < \gamma < 2^\ast/2$,
        \[
        h(t) \ge S^{-2^\ast\!/2}\, t^{2^\ast\!/2-1} + \eta\, t^{\gamma-1} \quad \forall t \ge 0;
        \]
    \item[$(ii)$] for some constant $b > S^{-2^\ast\!/2}$,
        \[
        h(t) \ge bt^{2^\ast\!/2-1} \quad \forall t \ge 0;
        \]
    \item[$(iii)$] for some constants $b > 0$ and $\gamma > 2^\ast/2$,
        \[
        h(t) > \max \set{S^{-2^\ast\!/2}\, t^{2^\ast\!/2-1},bt^{\gamma-1}} \quad \forall t > 0.
        \]
    \end{enumerate}
\end{properties}
We have the following theorem.

\begin{theorem} \label{Theorem 2.4}
Assume that \eqref{1.2} and \ref{A4} hold. Then $J$ is bounded from below, coercive, and satisfies the {\em \PS{c}} condition for all $c \in \R$. In particular, $J$ has a global minimizer.
\end{theorem}

\begin{proof}
By \eqref{1.2} and \eqref{1.7},
\[
J(u) \ge \half\, H\bigg(\int_\Omega |\nabla u|^2\, dx\bigg) - c_3 \int_\Omega |u|^p\, dx - c_4 - \frac{1}{2^\ast}\, S^{-2^\ast\!/2} \left(\int_\Omega |\nabla u|^2\, dx\right)^{2^\ast\!/2}
\]
for some constants $c_3, c_4 > 0$. By \ref{A4}, $H$ satisfies one of the following:
\begin{enumroman}
\item $H(t) \ge \dfrac{2}{2^\ast}\, S^{-2^\ast\!/2}\, t^{2^\ast\!/2} + \dfrac{\eta}{\gamma}\, t^\gamma$ for all $t \ge 0$, where $\eta > 0$ and $p/2 < \gamma < 2^\ast/2$;
\item $H(t) \ge \dfrac{2b}{2^\ast}\, t^{2^\ast\!/2}$ for all $t \ge 0$, where $b > S^{-2^\ast\!/2}$;
\item $H(t) \ge \dfrac{b}{\gamma}\, t^\gamma$ for all $t \ge 0$, where $b > 0$ and $\gamma > 2^\ast/2$.
\end{enumroman}
It follows that $J$ is bounded from below and coercive.

Let $c \in \R$ and let $\seq{u_j}$ be a \PS{c} sequence. By coercivity, $\seq{u_j}$ is bounded. By \ref{A4}, $h(t) > S^{-2^\ast\!/2}\, t^{2^\ast\!/2-1}$ for all $t > 0$, so $I = \emptyset$. So $\seq{u_j}$ has a strongly convergent subsequence by Proposition \ref{Proposition 2.2}.
\end{proof}

Finally we apply our results to the model problem
\begin{equation} \label{1.11}
\left\{\begin{aligned}
- \left[a + b \left(\int_\Omega |\nabla u|^2\, dx\right)^{\gamma-1}\right] \Delta u & = f(x,u) + |u|^{2^\ast-2}\, u && \text{in } \Omega\\[5pt]
u & = 0 && \text{on } \bdry{\Omega},
\end{aligned}\right.
\end{equation}
where $a, b \ge 0$ with $a + b > 0$ and $1 < \gamma < + \infty$. Here
\[
h(t) = a + bt^{\gamma-1}, \qquad H(t) = at + \frac{b}{\gamma}\, t^\gamma, \qquad K(t) = \frac{1}{N}\, at + \left(\frac{1}{2 \gamma} - \frac{1}{2^\ast}\right) bt^\gamma
\]
and
\[
J(u) = \frac{a}{2} \int_\Omega |\nabla u|^2\, dx + \frac{b}{2 \gamma} \left(\int_\Omega |\nabla u|^2\, dx\right)^\gamma - \int_\Omega F(x,u)\, dx - \frac{1}{2^\ast} \int_\Omega |u|^{2^\ast} dx, \quad u \in H^1_0(\Omega).
\]

Theorem \ref{Theorem 2.3} has the following corollary for the case $\gamma < 2^\ast/2$.

\begin{corollary} \label{Corollary 1.6}
Let $1 < \gamma < 2^\ast/2$ and $a, b \ge 0$. Assume that $f$ satisfies \eqref{1.2} and
\[
F(x,t) - \frac{1}{2^\ast}\, tf(x,t) \le \frac{1}{N}\, \lambda at^2 + \left(\frac{1}{2 \gamma} - \frac{1}{2^\ast}\right) \mu b\, |t|^{2\gamma} \quad \text{for a.a.\! } x \in \Omega \text{ and all } t \in \R
\]
for some constants $\lambda \le \lambda_1$ and $\mu \le \lambda_1(\gamma)$ with either $a > 0$ and $\lambda < \lambda_1$, or $b > 0$ and $\mu < \lambda_1(\gamma)$. Let $t_0$ be the unique positive solution of the equation
\[
a + bt^{\gamma-1} = S^{-2^\ast\!/2}\, t^{2^\ast\!/2-1}
\]
and set
\[
c^\ast = \frac{1}{N}\, at_0 + \left(\frac{1}{2 \gamma} - \frac{1}{2^\ast}\right) bt_0^\gamma.
\]
Then $J$ satisfies the {\em \PS{c}} condition for all $c < c^\ast$.
\end{corollary}

\begin{remark}
For $a = 1$ and $b = 0$, Corollary \ref{Corollary 1.6} gives the well-known compactness threshold
\[
c^\ast = \frac{1}{N}\, S^{N/2}
\]
in the Br{\'e}zis-Nirenberg problem (see \cite{MR709644}).
\end{remark}

\begin{remark}
An interesting special case of problem \eqref{1.11} is
\[
\left\{\begin{aligned}
- \bigg(\int_\Omega |\nabla u|^2\, dx\bigg)^{\gamma-1}\, \Delta u & = \mu\, |u|^{2\gamma-2}\, u + |u|^{2^\ast-2}\, u && \text{in } \Omega\\[5pt]
u & = 0 && \text{on } \bdry{\Omega},
\end{aligned}\right.
\]
where $\gamma > 1$ and $\mu > 0$. For $\gamma < 2^\ast/2$ and $\mu < \lambda_1(\gamma)$, Corollary \ref{Corollary 1.6} gives the compactness threshold
\[
c^\ast = \left(\frac{1}{2 \gamma} - \frac{1}{2^\ast}\right) S^{2^\ast \gamma/(2^\ast - 2 \gamma)}
\]
for the associated variational functional
\[
J(u) = \frac{1}{2 \gamma} \left(\int_\Omega |\nabla u|^2\, dx\right)^\gamma - \frac{\mu}{2 \gamma} \int_\Omega |u|^{2 \gamma}\, dx - \frac{1}{2^\ast} \int_\Omega |u|^{2^\ast} dx, \quad u \in H^1_0(\Omega).
\]
\end{remark}

\begin{remark}
The classical case $h(t) = a + bt$ when $N = 3$ is one of the few cases with both $a$ and $b$ positive and $\gamma < 2^\ast/2$ where $c^\ast$ can be found in closed form. Here Corollary \ref{Corollary 1.6} gives
\[
c^\ast = \frac{1}{4}\, abS^3 + \frac{1}{24}\, b^3 S^6 + \frac{1}{24} \left(4aS + b^2 S^4\right)^{3/2}.
\]
This threshold level was also obtained in Naimen \cite[Lemma 2.5]{MR3278854} using concentration compactness arguments. Our approach here is simpler.
\end{remark}

Theorem \ref{Theorem 2.3} also has the following corollary for the case $\gamma = 2^\ast/2$.

\begin{corollary} \label{Corollary 1.7}
Let $\gamma = 2^\ast/2$, $a > 0$, and $0 < b < S^{-2^\ast\!/2}$. Assume that $f$ satisfies \eqref{1.2} and
\[
F(x,t) - \frac{1}{2^\ast}\, tf(x,t) \le \frac{1}{N}\, \lambda at^2 \quad \text{for a.a.\! } x \in \Omega \text{ and all } t \in \R
\]
for some constant $\lambda < \lambda_1$. Set
\[
c^\ast = \frac{1}{N} \left(\frac{a^{2^\ast\!/2}}{S^{-2^\ast\!/2} - b}\right)^{2/(2^\ast-2)}.
\]
Then $J$ satisfies the {\em \PS{c}} condition for all $c < c^\ast$.
\end{corollary}

\begin{remark}
For the classical case $h(t) = a + bt$ with $N = 4$, $a > 0$, and $0 < b < S^{-2}$, Corollary \ref{Corollary 1.7} gives
\[
c^\ast = \frac{a^2}{4\, (S^{-2} - b)}.
\]
This threshold level was also obtained in Naimen \cite[Lemma 2.1]{MR3210026} by analyzing the behavior of Palais-Smale sequences. Our approach here is simpler again.
\end{remark}

Theorem \ref{Theorem 2.4} has the following corollary for $\gamma \ge 2^\ast/2$.

\begin{corollary} \label{Corollary 2.11}
Assume that $f$ satisfies \eqref{1.2}. Then $J$ is bounded from below, satisfies the {\em \PS{c}} condition for all $c \in \R$, and has a global minimizer in each of the following cases:
\begin{enumroman}
\item \label{Corollary 2.11.i} $\gamma = 2^\ast/2$, $a \ge 0$, and $b > S^{-2^\ast\!/2}$;
\item \label{Corollary 2.11.ii} $\gamma > 2^\ast/2$ and
    \[
    a^{\gamma-2^\ast\!/2}\, b^{2^\ast\!/2-1} > \frac{(\gamma - 2^\ast/2)^{\gamma-2^\ast\!/2}\, (2^\ast/2 - 1)^{2^\ast\!/2-1}}{(\gamma - 1)^{\gamma-1}}\, S^{-(2^\ast\!/2)(\gamma-1)}.
    \]
\end{enumroman}
\end{corollary}

\begin{proof}
The minimum of $a + bt^{\gamma-1} - S^{-2^\ast\!/2}\, t^{2^\ast\!/2-1},\, t > 0$ is positive if and only if the last inequality holds.
\end{proof}

\begin{remark}
For the classical case $h(t) = a + bt$ with $N = 4$, $a \ge 0$, and $b > S^{-2}$, Corollary \ref{Corollary 2.11} \ref{Corollary 2.11.i} implies that $J$ satisfies the \PS{c} condition for all $c \in \R$. This was also observed in Liao et al.\! \cite[Proposition 2.1]{MR3861499}.
\end{remark}

\begin{remark}
For the classical case $h(t) = a + bt$ with $N \ge 5$, Corollary \ref{Corollary 2.11} \ref{Corollary 2.11.ii} implies that $J$ satisfies the \PS{c} condition for all $c \in \R$ if
\[
a^{N-4}\, b^2 > \frac{4\, (N - 4)^{N-4}}{(N - 2)^{N-2}}\, S^{-N}.
\]
\end{remark}

\section{Existence and multiplicity results} \label{Section 3}

In the case where $I$ is nonempty, our main existence result for problem \eqref{1.1} is the following theorem.

\begin{theorem} \label{Theorem 1.10}
Assume \eqref{1.2} and \ref{A1}--\ref{A3}. Assume further that
\begin{equation} \label{1.12}
H(t) \ge b_0 t^{\gamma_0} \quad \text{for } 0 \le t \le \delta
\end{equation}
for some constants $\delta, b_0 > 0$ and $1 \le \gamma_0 < 2^\ast/2$,
\begin{equation} \label{1.13}
F(x,t) \le \half\, \mu_0 b_0\, |t|^{2\gamma_0} \quad \text{for a.a.\! } x \in \Omega \text{ and } |t| \le \delta
\end{equation}
for some $\mu_0 < \lambda_1(\gamma_0)$, and
\begin{equation} \label{1.14}
F(x,t) \ge \frac{1}{q}\, \nu t^q \quad \text{for a.a.\! } x \in B_r(x_0) \text{ and all } t \ge 0
\end{equation}
for some ball $B_r(x_0) \subset \Omega$, $\nu > 0$, and $2 \gamma_0 \le q \le 2 \gamma_n$. Then problem \eqref{1.1} has a nontrivial solution in each of the following cases:
\begin{enumroman}
\item $N = 3$ and $q > 4$,
\item $N \ge 4$ and $q \ge N/(N - 2)$.
\end{enumroman}
\end{theorem}

We will show that the functional $J$ has the mountain pass geometry and the mountain pass level is below the compactness threshold $c^\ast$ in \eqref{2.17}.

\begin{lemma} \label{Lemma 2.1}
If \eqref{1.2}, \eqref{1.12}, and \eqref{1.13} hold, then $\exists \rho > 0$ such that
\begin{equation} \label{2.7}
\inf_{\norm{u} = \rho}\, J(u) > 0.
\end{equation}
\end{lemma}

\begin{proof}
By \eqref{1.2} and \eqref{1.13},
\[
F(x,t) \le \half\, \mu_0 b_0\, |t|^{2\gamma_0} + c_5 |t|^{2^\ast} \quad \text{for a.a.\! } x \in \Omega \text{ and all } t \in \R
\]
for some constant $c_5 > 0$. This together with \eqref{1.12} implies that for $\norm{u} \le \sqrt{\delta}$,
\begin{multline*}
J(u) \ge \half\, b_0 \left[\left(\int_\Omega |\nabla u|^2\, dx\right)^{\gamma_0} - \mu_0 \int_\Omega |u|^{2\gamma_0}\, dx\right] - \left(c_5 + \frac{1}{2^\ast}\right) \int_\Omega |u|^{2^\ast} dx\\[10pt]
\ge \half\, b_0 \left(1 - \frac{\mu_0}{\lambda_1(\gamma_0)} + \o(1)\right) \norm{u}^{2\gamma_0} \quad \text{as } \norm{u} \to 0
\end{multline*}
since $2^\ast > 2 \gamma_0$. Since $\mu_0 < \lambda_1(\gamma_0)$, the desired conclusion follows from this.
\end{proof}

Next we show that for a suitably chosen $v \in H^1_0(\Omega) \setminus \set{0}$, $J(sv) \to - \infty$ as $s \to + \infty$ and the maximum of $J$ on the ray $sv,\, s \ge 0$ is strictly less than $c^\ast$. Take a function $\psi \in C^\infty_0(B_r(x_0))$ such that $0 \le \psi \le 1$ on $B_r(x_0)$ and $\psi = 1$ on $B_{r/2}(x_0)$, and set
\[
u_\eps(x) = \frac{\psi(x)}{\left(\eps + |x - x_0|^2\right)^{(N-2)/2}}
\]
and
\[
v_\eps = \frac{u_\eps}{\pnorm[2^\ast]{u_\eps}}
\]
for $\eps > 0$. Then
\begin{equation} \label{2.9}
\int_\Omega |\nabla v_\eps|^2\, dx = S + \O\big(\eps^{(N-2)/2}\big)
\end{equation}
and
\begin{equation} \label{2.10}
\int_\Omega v_\eps^q\, dx = \begin{cases}
\kappa \eps^{(2N-(N-2)q)/4} + \O\big(\eps^{(N-2)q/4}\big) & \text{if } q > N/(N - 2)\\[10pt]
\kappa \eps^{N/4}\, |\log \eps| + \O\big(\eps^{N/4}\big) & \text{if } q = N/(N - 2)
\end{cases}
\end{equation}
for some constant $\kappa > 0$ (see, e.g., Dr{\'a}bek and Huang \cite{MR1473856}).

\begin{lemma}
For all sufficiently small $\eps > 0$,
\begin{equation} \label{2.16}
J(sv_\eps) \to - \infty \quad \text{as } s \to + \infty.
\end{equation}
\end{lemma}

\begin{proof}
Since $\pnorm[2^\ast]{v_\eps} = 1$, $\pnorm[p]{v_\eps}$ is bounded and \eqref{1.2} gives
\begin{equation} \label{2.8}
J(sv_\eps) \le \half\, H\bigg(s^2 \int_\Omega |\nabla v_\eps|^2\, dx\bigg) + c_6 s^p + c_7 - \frac{s^{2^\ast}}{2^\ast}, \quad s \ge 0
\end{equation}
for some constants $c_6, c_7 > 0$. Set
\[
t = s^2 \int_\Omega |\nabla v_\eps|^2\, dx.
\]
Then $t \to + \infty$ as $s \to + \infty$ and \eqref{2.8} gives
\begin{equation} \label{2.11}
J(sv_\eps) \le \half\, H(t) + c_6 t^{p/2} \left(\int_\Omega |\nabla v_\eps|^2\, dx\right)^{-p/2} + c_7 - \frac{t^{2^\ast\!/2}}{2^\ast} \left(\int_\Omega |\nabla v_\eps|^2\, dx\right)^{-2^\ast\!/2}.
\end{equation}
By \ref{A3},
\[
\lim_{t \to + \infty}\, \frac{H(t)}{t^{2^\ast\!/2}} = \frac{2b}{2^\ast},
\]
so \eqref{2.11} gives
\[
J(sv_\eps) \le c_6 t^{p/2} \left(\int_\Omega |\nabla v_\eps|^2\, dx\right)^{-p/2} + c_7 - \frac{t^{2^\ast\!/2}}{2^\ast} \left[\left(\int_\Omega |\nabla v_\eps|^2\, dx\right)^{-2^\ast\!/2} - b + \o(1)\right]
\]
as $t \to + \infty$. Since $\int_\Omega |\nabla v_\eps|^2\, dx \to S$ as $\eps \to 0$ by \eqref{2.9}, $b < S^{-2^\ast\!/2}$, and $p < 2^\ast$, the desired conclusion follows.
\end{proof}

\begin{lemma}
In each of the two cases in Theorem \ref{Theorem 1.10},
\begin{equation} \label{2.12}
\max_{s \ge 0}\, J(sv_\eps) < c^\ast
\end{equation}
for all sufficiently small $\eps > 0$.
\end{lemma}

\begin{proof}
Since $v_\eps = 0$ outside $B_r(x_0)$, \eqref{1.14} gives
\[
J(sv_\eps) \le \half\, H\bigg(s^2 \int_\Omega |\nabla v_\eps|^2\, dx\bigg) - \frac{1}{q}\, \nu s^q \int_\Omega v_\eps^q\, dx - \frac{s^{2^\ast}}{2^\ast} =: z_\eps(s),
\]
so it suffices to show that
\[
\max_{s \ge 0}\, z_\eps(s) < c^\ast
\]
for sufficiently small $\eps > 0$. Suppose this is false. Then there are sequences $\seq{\eps_j}$ and $\seq{s_j}$, with $\eps_j, s_j > 0$ and $\eps_j \to 0$, such that
\begin{equation} \label{2.13}
z_{\eps_j}(s_j) = \half\, H\bigg(s_j^2 \int_\Omega |\nabla v_{\eps_j}|^2\, dx\bigg) - \frac{1}{q}\, \nu s_j^q \int_\Omega v_{\eps_j}^q\, dx - \frac{s_j^{2^\ast}}{2^\ast} \ge c^\ast
\end{equation}
and
\begin{equation} \label{2.14}
s_j\, z_{\eps_j}'(s_j) = h\bigg(s_j^2 \int_\Omega |\nabla v_{\eps_j}|^2\, dx\bigg)\, s_j^2 \int_\Omega |\nabla v_{\eps_j}|^2\, dx - \nu s_j^q \int_\Omega v_{\eps_j}^q\, dx - s_j^{2^\ast} = 0.
\end{equation}
Set
\[
t_j = s_j^2 \int_\Omega |\nabla v_{\eps_j}|^2\, dx.
\]
Then \eqref{2.14} gives
\begin{equation} \label{2.15}
\frac{h(t_j)}{t_j^{2^\ast\!/2-1}} = \frac{1}{\left(\dint_\Omega |\nabla v_{\eps_j}|^2\, dx\right)^{2^\ast\!/2}} + \nu t_j^{-(2^\ast-q)/2}\, \frac{\dint_\Omega v_{\eps_j}^q\, dx}{\left(\dint_\Omega |\nabla v_{\eps_j}|^2\, dx\right)^{q/2}}.
\end{equation}
If $t_j \to + \infty$ for a renamed subsequence, then the left-hand side goes to $b$ by \ref{A3}, while the right-hand side goes to $S^{-2^\ast\!/2}$ since $\int_\Omega |\nabla v_{\eps_j}|^2\, dx \to S$ by \eqref{2.9} and $\int_\Omega v_{\eps_j}^q\, dx \to 0$ by \eqref{2.10}, contradicting our assumption that $b < S^{-2^\ast\!/2}$. So $\seq{t_j}$ is bounded, and hence converges to some $t \ge 0$ for a renamed subsequence. Then $s_j^2 \to S^{-1}\, t$ and hence passing to the limit in \eqref{2.13} gives
\[
\half\, H(t) - \frac{1}{2^\ast}\, S^{-2^\ast\!/2}\, t^{2^\ast\!/2} > 0
\]
since $c^\ast > 0$, so $t > 0$. On the other hand, passing to the limit in \eqref{2.14} shows that $t$ satisfies \eqref{1.9}. Since $t_0$ is the unique positive solution of this equation, it follows that $t = t_0$.

Now combining \eqref{2.15} with \eqref{2.9} and \eqref{2.10} gives
\[
\frac{h(t_j)}{t_j^{2^\ast\!/2-1}} = S^{-2^\ast\!/2} + \begin{cases}
\sigma_j\, \eps_j^{(2N-(N-2)q)/4} + \O\big(\eps_j^{(N-2)/2}\big) & \text{if } q > N/(N - 2)\\[10pt]
\sigma_j\, \eps_j^{N/4}\, |\log \eps_j| + \O\big(\eps_j^{\min \set{(N-2)/2,N/4}}\big) & \text{if } q = N/(N - 2),
\end{cases}
\]
where $\sigma_j \to \kappa \nu S^{-q/2}\, t_0^{-(2^\ast-q)/2} > 0$. It follows from this that in each of the two cases in the lemma,
\[
\frac{h(t_j)}{t_j^{2^\ast\!/2-1}} \ge S^{-2^\ast\!/2} = \frac{h(t_0)}{t_0^{2^\ast\!/2-1}}
\]
for all sufficiently large $j$. Then $t_j \le t_0$ by \ref{A3}. Since $K$ is nondecreasing, then
\[
K(t_j) \le K(t_0) = c^\ast.
\]
However, dividing \eqref{2.14} by $2^\ast$ and subtracting from \eqref{2.13} gives
\[
K(t_j) - \left(\frac{1}{q} - \frac{1}{2^\ast}\right) \nu s_j^q \int_\Omega v_{\eps_j}^q\, dx \ge c^\ast,
\]
so $K(t_j) > c^\ast$. This contradiction completes the proof.
\end{proof}

We are now ready to prove Theorem \ref{Theorem 1.10}.

\begin{proof}[Proof of Theorem \ref{Theorem 1.10}]
Let $\rho$ be as in Lemma \ref{Lemma 2.1} and fix $\eps > 0$ such that \eqref{2.16} and \eqref{2.12} hold. Then $\exists R > \rho$ such that $J(Rv_\eps) \le 0$. Let
\[
\Gamma = \set{\varphi \in C([0,1],H^1_0(\Omega)) : \varphi(0) = 0,\, \varphi(1) = Rv_\eps}
\]
be the class of paths in $H^1_0(\Omega)$ joining the origin to $Rv_\eps$, and set
\[
c := \inf_{\varphi \in \Gamma}\, \max_{u \in \varphi([0,1])}\, J(u).
\]
By \eqref{2.7}, $c > 0$. Since the path $\varphi_0(s) = sRv_\eps,\, s \in [0,1]$ is in $\Gamma$,
\[
c \le \max_{u \in \varphi_0([0,1])}\, J(u) \le \max_{s \ge 0}\, J(sv_\eps) < c^\ast,
\]
so $J$ satisfies the \PS{c} condition by Theorem \ref{Theorem 2.3}. Hence $J$ has a critical point $u$ with $J(u) = c$ by the mountain pass theorem (see Ambrosetti and Rabinowitz \cite{MR0370183}). Then $u$ is a weak solution of problem \eqref{1.1} and $u$ is nontrivial since $c > 0$.
\end{proof}

Theorem \ref{Theorem 1.10} has many interesting consequences, some of which we now present. First we consider the problem
\begin{equation} \label{1.15}
\left\{\begin{aligned}
- h\bigg(\int_\Omega |\nabla u|^2\, dx\bigg)\, \Delta u & = \lambda u + |u|^{2^\ast-2}\, u && \text{in } \Omega\\[5pt]
u & = 0 && \text{on } \bdry{\Omega},
\end{aligned}\right.
\end{equation}
where $\lambda > 0$. Assume that
\begin{equation} \label{3.15}
K(t) \ge \alpha t \quad \forall t \ge 0
\end{equation}
for some constant $\alpha > 0$ and
\begin{equation} \label{1.16}
H(t) \ge a_0\, t \quad \text{for } 0 \le t \le \delta
\end{equation}
for some constants $\delta, a_0 > 0$. We have $f(x,t) = \lambda t$ and
\[
F(x,t) = \half\, \lambda t^2, \qquad F(x,t) - \frac{1}{2^\ast}\, tf(x,t) = \frac{1}{N}\, \lambda t^2,
\]
so \ref{A1} holds with $\mu_1 = \lambda/N \alpha$ if $\lambda < N \alpha \lambda_1$, \eqref{1.13} holds with $\gamma_0 = 1$ and $\mu_0 = \lambda/a_0$ if $\lambda < a_0 \lambda_1$, and \eqref{1.14} holds with $q = 2$ if $\lambda > 0$. So Theorem \ref{Theorem 1.10} has the following corollary for problem \eqref{1.15}.

\begin{corollary} \label{Theorem 1.11}
Assume \eqref{3.15}, \ref{A2}, \ref{A3}, and \eqref{1.16}. If
\[
0 < \lambda < \min \set{a_0,N \alpha} \lambda_1
\]
and $N \ge 4$, then problem \eqref{1.15} has a nontrivial solution.
\end{corollary}

In particular, we have the following corollary in the model case $h(t) = a + bt^{\gamma-1}$.

\begin{corollary} \label{Corollary 3.6}
The problem
\[
\left\{\begin{aligned}
- \left[a + b \left(\int_\Omega |\nabla u|^2\, dx\right)^{\gamma-1}\right] \Delta u & = \lambda u + |u|^{2^\ast-2}\, u && \text{in } \Omega\\[5pt]
u & = 0 && \text{on } \bdry{\Omega},
\end{aligned}\right.
\]
where either $1 < \gamma < 2^\ast/2$ and $b \ge 0$, or $\gamma = 2^\ast/2$ and $0 \le b < S^{-2^\ast\!/2}$, has a nontrivial solution if $0 < \lambda < a \lambda_1$ and $N \ge 4$.
\end{corollary}

Next we consider the problem
\begin{equation} \label{1.17}
\left\{\begin{aligned}
- h\bigg(\int_\Omega |\nabla u|^2\, dx\bigg)\, \Delta u & = \mu\, |u|^{2\gamma-2}\, u + |u|^{2^\ast-2}\, u && \text{in } \Omega\\[5pt]
u & = 0 && \text{on } \bdry{\Omega},
\end{aligned}\right.
\end{equation}
where $\mu > 0$ and $1 < \gamma < 2^\ast/2$. Assume that
\begin{equation} \label{3.18}
K(t) \ge \beta t^\gamma \quad \forall t \ge 0
\end{equation}
for some constant $\beta > 0$ and
\begin{equation} \label{1.18}
H(t) \ge b_0 t^\gamma \quad \text{for } 0 \le t \le \delta
\end{equation}
for some constants $\delta, b_0 > 0$. We have $f(x,t) = \mu\, |t|^{2\gamma-2}\, t$ and
\[
F(x,t) = \frac{1}{2 \gamma}\, \mu\, |t|^{2\gamma}, \qquad F(x,t) - \frac{1}{2^\ast}\, tf(x,t) = \left(\frac{1}{2 \gamma} - \frac{1}{2^\ast}\right) \mu\, |t|^{2\gamma},
\]
so \ref{A1} holds with $\mu_1 = (1/2 \gamma - 1/2^\ast)\, \mu/\beta$ if
\[
\mu < \left(\frac{1}{2 \gamma} - \frac{1}{2^\ast}\right)^{-1} \beta \lambda_1(\gamma),
\]
\eqref{1.13} holds with $\gamma_0 = \gamma$ and $\mu_0 = \mu/\gamma b_0$ if $\mu < \gamma b_0 \lambda_1(\gamma)$, and \eqref{1.14} holds with $q = 2 \gamma$ if $\mu > 0$. So Theorem \ref{Theorem 1.10} has the following corollary for problem \eqref{1.17}.

\begin{corollary} \label{Theorem 1.12}
Assume \eqref{3.18}, \ref{A2}, \ref{A3}, and \eqref{1.18}. If
\[
0 < \mu < \min \set{\gamma b_0,\left(\frac{1}{2 \gamma} - \frac{1}{2^\ast}\right)^{-1} \beta} \lambda_1(\gamma)
\]
and $N \ge 4$, or $N = 3$ and $\gamma > 2$, then problem \eqref{1.17} has a nontrivial solution.
\end{corollary}

In particular, we have the following corollary in the model case $h(t) = a + bt^{\gamma-1}$.

\begin{corollary}
The problem
\[
\left\{\begin{aligned}
- \left[a + b \left(\int_\Omega |\nabla u|^2\, dx\right)^{\gamma-1}\right] \Delta u & = \mu\, |u|^{2\gamma-2}\, u + |u|^{2^\ast-2}\, u && \text{in } \Omega\\[5pt]
u & = 0 && \text{on } \bdry{\Omega},
\end{aligned}\right.
\]
where $a \ge 0$ and $1 < \gamma < 2^\ast/2$, has a nontrivial solution if $0 < \mu < b \lambda_1(\gamma)$ and $N \ge 4$, or $N = 3$ and $\gamma > 2$.
\end{corollary}

Finally we consider the problem
\begin{equation} \label{1.19}
\left\{\begin{aligned}
- h\bigg(\int_\Omega |\nabla u|^2\, dx\bigg)\, \Delta u & = \nu\, |u|^{q-2}\, u + |u|^{2^\ast-2}\, u && \text{in } \Omega\\[5pt]
u & = 0 && \text{on } \bdry{\Omega},
\end{aligned}\right.
\end{equation}
where $\nu > 0$ and $2 < q < 2^\ast$. Assume that for some constants $\alpha, \beta > 0$ and $q/2 < \gamma < 2^\ast/2$,
\begin{equation} \label{3.21}
K(t) \ge \alpha t + \beta t^\gamma \quad \forall t \ge 0.
\end{equation}
Since $h$ is nonnegative, $H(t) \ge 2K(t) \ge 2 \alpha t$, so \eqref{1.12} holds with $b_0 = 2 \alpha$ and $\gamma_0 = 1$. We have $f(x,t) = \nu\, |t|^{q-2}\, t$ and
\[
F(x,t) = \frac{1}{q}\, \nu\, |t|^q, \qquad F(x,t) - \frac{1}{2^\ast}\, tf(x,t) = \left(\frac{1}{q} - \frac{1}{2^\ast}\right) \nu\, |t|^q.
\]
Since $q > 2$, \eqref{1.13} holds for any $\mu_0 > 0$ if $\delta > 0$ is sufficiently small. Theorem \ref{Theorem 1.10} has the following corollary for problem \eqref{1.19}.

\begin{corollary}
Assume \eqref{3.21}, \ref{A2}, and \ref{A3}. If
\[
0 < \nu < (2 \gamma - 2) \left(\frac{1}{q} - \frac{1}{2^\ast}\right)^{-1} \left(\frac{\alpha \lambda_1}{2 \gamma - q}\right)^{(2\gamma-q)/(2\gamma-2)} \left(\frac{\beta \lambda_1(\gamma)}{q - 2}\right)^{(q-2)/(2\gamma-2)}
\]
and $N \ge 4$, or $N = 3$ and $q > 4$, then problem \eqref{1.19} has a nontrivial solution.
\end{corollary}

\begin{proof}
To see that \ref{A1} holds, note that the minimum of
\[
\lambda \alpha t^2 + \mu \beta |t|^{2\gamma} - \left(\frac{1}{q} - \frac{1}{2^\ast}\right) \nu\, |t|^q, \quad t \in \R
\]
is nonnegative if and only if
\[
\nu \le (2 \gamma - 2) \left(\frac{1}{q} - \frac{1}{2^\ast}\right)^{-1} \left(\frac{\alpha \lambda}{2 \gamma - q}\right)^{(2\gamma-q)/(2\gamma-2)} \left(\frac{\beta \mu}{q - 2}\right)^{(q-2)/(2\gamma-2)}. \QED
\]
\end{proof}

In the case where $I$ is empty, first we consider the model problem
\begin{equation} \label{3.22}
\left\{\begin{aligned}
- \left[a + b \left(\int_\Omega |\nabla u|^2\, dx\right)^{\gamma-1}\right] \Delta u & = \lambda u + g(x,u) + |u|^{2^\ast-2}\, u && \text{in } \Omega\\[5pt]
u & = 0 && \text{on } \bdry{\Omega},
\end{aligned}\right.
\end{equation}
where $a, b \ge 0$ and $2^\ast/2 \le \gamma < + \infty$ satisfy one of the two conditions in Corollary \ref{Corollary 2.11}, $\lambda > 0$, and $g$ is a Carath\'{e}odory function on $\Omega \times \R$ satisfying
\begin{equation} \label{3.23}
g(x,t) = \o(t) \quad \text{as } t \to 0, \text{uniformly a.e.\! in } \Omega
\end{equation}
and
\begin{equation} \label{3.24}
|g(x,t)| \le c_8 |t|^{p-1} + c_9 \quad \text{for a.a.\! } x \in \Omega \text{ and all } t \in \R
\end{equation}
for some constants $c_8, c_9 > 0$ and $2 < p < 2^\ast$. The associated variational functional is
\begin{multline*}
J(u) = \frac{a}{2} \int_\Omega |\nabla u|^2\, dx + \frac{b}{2 \gamma} \left(\int_\Omega |\nabla u|^2\, dx\right)^\gamma - \frac{\lambda}{2} \int_\Omega u^2\, dx - \int_\Omega G(x,u)\, dx - \frac{1}{2^\ast} \int_\Omega |u|^{2^\ast} dx,\\[7.5pt]
u \in H^1_0(\Omega),
\end{multline*}
where $G(x,t) = \int_0^t g(x,s)\, ds$ is the primitive of $g$. We note that
\begin{equation} \label{3.30}
\int_\Omega G(x,u)\, dx = \o(\norm{u}^2) \quad \text{as } \norm{u} \to 0
\end{equation}
by \eqref{3.23} and \eqref{3.24}.

When $a = 0$, we have the following existence result.

\begin{theorem} \label{Theorem 3.10}
Assume that $g$ satisfies \eqref{3.23} and \eqref{3.24}. If $\gamma = 2^\ast/2$, $a = 0$, and $b > S^{-2^\ast\!/2}$, then problem \eqref{3.22} has a nontrivial solution for all $\lambda > 0$.
\end{theorem}

\begin{proof}
By Corollary \ref{Corollary 2.11}, $J$ has a global minimizer $u_0$. For any $u \in H^1_0(\Omega) \setminus \set{0}$,
\[
J(su) = - \frac{\lambda s^2}{2} \int_\Omega u^2\, dx + \o(s^2) \quad \text{as } s \to 0
\]
by \eqref{3.30}, so $J(su) < 0$ if $s > 0$ is sufficiently small. So $J(u_0) = \inf_{H^1_0(\Omega)}\, J < 0$ and hence $u_0$ is nontrivial.
\end{proof}

When $a > 0$, we prove a multiplicity result. Let $0 < \lambda_1 < \lambda_2 \le \lambda_3 \le \cdots$ be the Dirichlet eigenvalues of $- \Delta$ on $\Omega$, repeated according to multiplicity.

\begin{theorem} \label{Theorem 3.11}
Assume that $g$ satisfies \eqref{3.23} and \eqref{3.24}.
\begin{enumroman}
\item \label{Theorem 3.11.i} If $\gamma = 2^\ast/2$, $a > 0$, and $b > S^{-2^\ast\!/2}$, then problem \eqref{3.22} has at least two nontrivial solutions in each of the following cases:
    \begin{enumerate}
    \item[$(a)$] $a \lambda_k < \lambda < a \lambda_{k+1}$ for some $k \ge 1$;
    \item[$(b)$] $a \lambda_k < \lambda = a \lambda_{k+1}$ for some $k \ge 1$ and $G(x,t) \le 0$ for a.a.\! $x \in \Omega$ and $|t| \le \delta$ for some $\delta > 0$.
    \end{enumerate}
\item \label{Theorem 3.11.ii} If $\gamma > 2^\ast/2$ and
    \[
    a^{\gamma-2^\ast\!/2}\, b^{2^\ast\!/2-1} > \frac{(\gamma - 2^\ast/2)^{\gamma-2^\ast\!/2}\, (2^\ast/2 - 1)^{2^\ast\!/2-1}}{(\gamma - 1)^{\gamma-1}}\, S^{-(2^\ast\!/2)(\gamma-1)},
    \]
    then problem \eqref{3.22} has at least two nontrivial solutions in each of the following cases:
    \begin{enumerate}
    \item[$(a)$] $a \lambda_k < \lambda < a \lambda_{k+1}$ for some $k \ge 1$;
    \item[$(b)$] $a \lambda_k = \lambda < a \lambda_{k+1}$ for some $k \ge 1$ and $G(x,t) \ge 0$ for a.a.\! $x \in \Omega$ and $|t| \le \delta$ for some $\delta > 0$.
    \end{enumerate}
\end{enumroman}
\end{theorem}

We will prove this theorem using the following result of Brezis and Nirenberg \cite[Theorem 4]{MR1127041}.

\begin{proposition} \label{Proposition 3.12}
Let $J$ be a $C^1$-functional on a Banach space $X$. Assume that $J$ is bounded from below, $\inf_X\, J < 0$, and $J$ satisfies the {\em \PS{c}} condition for all $c \in \R$. Assume further that $X$ has a direct sum decomposition $X = V \oplus W,\, u = v + w$ with $\dim V < + \infty$ and
\[
\begin{cases}
J(v) \le 0 & \text{for } v \in V \cap B_r(0)\\[10pt]
J(w) \ge 0 & \text{for } w \in W \cap B_r(0)
\end{cases}
\]
for some $r > 0$. Then $J$ has at least two nontrivial critical points.
\end{proposition}

\begin{proof}[Proof of Theorem \ref{Theorem 3.11}]
By Corollary \ref{Corollary 2.11}, $J$ is bounded from below and satisfies the \PS{c} condition for all $c \in \R$. We have the direct sum decomposition $H^1_0(\Omega) = V \oplus W,\, u = v + w$, where $V$ is the span of the eigenfunctions associated with $\lambda_1, \dots, \lambda_k$ and $W$ is the orthogonal complement of $V$. For $v \in V$,
\begin{multline} \label{3.26}
J(v) \le - \half \left(\frac{\lambda}{\lambda_k} - a\right) \int_\Omega |\nabla v|^2\, dx + \frac{b}{2 \gamma} \left(\int_\Omega |\nabla v|^2\, dx\right)^\gamma - \int_\Omega G(x,v)\, dx - \frac{1}{2^\ast} \int_\Omega |v|^{2^\ast} dx\\[10pt]
= - \half \left(\frac{\lambda}{\lambda_k} - a\right) \int_\Omega |\nabla v|^2\, dx + \o(\norm{v}^2) \quad \text{as } \norm{v} \to 0
\end{multline}
by \eqref{3.30}, so $J(v) < 0$ if $\lambda > a \lambda_k$ and $\norm{v} > 0$ is sufficiently small. For $w \in W$,
\begin{multline*}
\hspace{-2.1pt} J(w) \ge \half \left(a - \frac{\lambda}{\lambda_{k+1}}\right) \int_\Omega |\nabla w|^2\, dx + \frac{b}{2 \gamma} \left(\int_\Omega |\nabla w|^2\, dx\right)^\gamma - \int_\Omega G(x,w)\, dx - \frac{1}{2^\ast} \int_\Omega |w|^{2^\ast} dx\\[10pt]
= \half \left(a - \frac{\lambda}{\lambda_{k+1}}\right) \int_\Omega |\nabla w|^2\, dx + \o(\norm{w}^2) \quad \text{as } \norm{w} \to 0
\end{multline*}
by \eqref{3.30}, so $J(w) \ge 0$ if $\lambda < a \lambda_{k+1}$ and $\norm{w}$ is sufficiently small. So $J$ has at least two nontrivial critical points by Proposition \ref{Proposition 3.12} in the cases \ref{Theorem 3.11.i}$(a)$ and \ref{Theorem 3.11.ii}$(a)$.

In the case \ref{Theorem 3.11.i}$(b)$, \eqref{1.7} gives
\[
J(w) \ge \int_\Omega \left[\frac{a}{2}\, \Big(|\nabla w|^2 - \lambda_{k+1}\, w^2\Big) - G(x,w)\right] dx + \frac{b - S^{-2^\ast\!/2}}{2^\ast} \left(\int_\Omega |\nabla w|^2\, dx\right)^{2^\ast\!/2} \quad \forall w \in W.
\]
The local sign condition on $G$ in this case implies that the first integral on the right-hand side is nonnegative if $\norm{w}$ is sufficiently small (see Li and Willem \cite{MR1312028}). Since $b > S^{-2^\ast\!/2}$, then $J(w) \ge 0$ when $\norm{w}$ is small. In the case \ref{Theorem 3.11.ii}$(b)$, \eqref{3.26} gives
\[
J(v) \le \frac{b}{2 \gamma} \left(\int_\Omega |\nabla v|^2\, dx\right)^\gamma - \int_{\set{|v| > \delta}} G(x,v)\, dx - \frac{1}{2^\ast} \int_\Omega |v|^{2^\ast} dx \quad \forall v \in V.
\]
Since $V$ is a finite dimensional subspace of $H^1_0(\Omega)$ consisting of $L^\infty$-functions and $\gamma > 2^\ast/2$, it follows from this that $J(v) < 0$ if $\norm{v} > 0$ is sufficiently small. So $J$ has two nontrivial critical points in these cases also.
\end{proof}

In the borderline case where $\gamma = 2^\ast/2$ and $b = S^{-2^\ast\!/2}$, lower-order terms come into play. We consider the problem
\begin{equation} \label{3.28}
\left\{\begin{aligned}
- h\bigg(\int_\Omega |\nabla u|^2\, dx\bigg)\, \Delta u & = \lambda u + g(x,u) + |u|^{2^\ast-2}\, u && \text{in } \Omega\\[5pt]
u & = 0 && \text{on } \bdry{\Omega},
\end{aligned}\right.
\end{equation}
where
\[
h(t) = a + S^{-2^\ast\!/2}\, t^{2^\ast\!/2-1} + \eta\, t^{\sigma-1}, \quad t \ge 0,
\]
$a \ge 0$, $\eta > 0$, $p/2 < \sigma < 2^\ast/2$, $\lambda > 0$, and $g$ is a Carath\'{e}odory function on $\Omega \times \R$ satisfying \eqref{3.23} and \eqref{3.24}. The associated functional is
\begin{multline*}
J(u) = \frac{a}{2} \int_\Omega |\nabla u|^2\, dx + \frac{S^{-2^\ast\!/2}}{2^\ast} \left(\int_\Omega |\nabla u|^2\, dx\right)^{2^\ast\!/2} + \frac{\eta}{2 \sigma} \left(\int_\Omega |\nabla u|^2\, dx\right)^\sigma\\[10pt]
- \frac{\lambda}{2} \int_\Omega u^2\, dx - \int_\Omega G(x,u)\, dx - \frac{1}{2^\ast} \int_\Omega |u|^{2^\ast} dx, \quad u \in H^1_0(\Omega),
\end{multline*}
where $G(x,t) = \int_0^t g(x,s)\, ds$ satisfies \eqref{3.30} as before. We have the following existence and multiplicity result.

\begin{theorem} \label{Theorem 3.13}
Let $\eta > 0$ and $p/2 < \sigma < 2^\ast/2$, and assume that $g$ satisfies \eqref{3.23} and \eqref{3.24}.
\begin{enumroman}
\item \label{Theorem 3.13.i} If $a = 0$, then problem \eqref{3.28} has a nontrivial solution for all $\lambda > 0$.
\item \label{Theorem 3.13.ii} If $a > 0$, then problem \eqref{3.28} has at least two nontrivial solutions in each of the following cases:
    \begin{enumerate}
    \item[$(a)$] $a \lambda_k < \lambda < a \lambda_{k+1}$ for some $k \ge 1$;
    \item[$(b)$] $a \lambda_k < \lambda = a \lambda_{k+1}$ for some $k \ge 1$ and $G(x,t) \le 0$ for a.a.\! $x \in \Omega$ and $|t| \le \delta$ for some $\delta > 0$.
    \end{enumerate}
\end{enumroman}
\end{theorem}

\begin{proof}
\ref{Theorem 3.13.i} By Theorem \ref{Theorem 2.4}, $J$ has a global minimizer $u_0$. For any $u \in H^1_0(\Omega) \setminus \set{0}$,
\[
J(su) = - \frac{\lambda s^2}{2} \int_\Omega u^2\, dx + \o(s^2) \quad \text{as } s \to 0
\]
by \eqref{3.30}, so $J(su) < 0$ if $s > 0$ is sufficiently small. So $J(u_0) = \inf_{H^1_0(\Omega)}\, J < 0$ and hence $u_0$ is nontrivial.

\ref{Theorem 3.13.ii} By Theorem \ref{Theorem 2.4}, $J$ is bounded from below and satisfies the \PS{c} condition for all $c \in \R$. We have the direct sum decomposition $H^1_0(\Omega) = V \oplus W,\, u = v + w$, where $V$ is the span of the eigenfunctions associated with $\lambda_1, \dots, \lambda_k$ and $W$ is the orthogonal complement of $V$. For $v \in V$,
\begin{multline*}
J(v) \le - \half \left(\frac{\lambda}{\lambda_k} - a\right) \int_\Omega |\nabla v|^2\, dx + \frac{S^{-2^\ast\!/2}}{2^\ast} \left(\int_\Omega |\nabla v|^2\, dx\right)^{2^\ast\!/2} + \frac{\eta}{2 \sigma} \left(\int_\Omega |\nabla v|^2\, dx\right)^\sigma\\[10pt]
- \int_\Omega G(x,v)\, dx - \frac{1}{2^\ast} \int_\Omega |v|^{2^\ast} dx = - \half \left(\frac{\lambda}{\lambda_k} - a\right) \int_\Omega |\nabla v|^2\, dx + \o(\norm{v}^2) \quad \text{as } \norm{v} \to 0
\end{multline*}
by \eqref{3.30}, so $J(v) < 0$ if $\lambda > a \lambda_k$ and $\norm{v} > 0$ is sufficiently small. For $w \in W$,
\begin{multline*}
J(w) \ge \half \left(a - \frac{\lambda}{\lambda_{k+1}}\right) \int_\Omega |\nabla w|^2\, dx + \frac{S^{-2^\ast\!/2}}{2^\ast} \left(\int_\Omega |\nabla w|^2\, dx\right)^{2^\ast\!/2} + \frac{\eta}{2 \sigma} \left(\int_\Omega |\nabla w|^2\, dx\right)^\sigma\\[10pt]
\hspace{-0.56pt} - \int_\Omega G(x,w)\, dx - \frac{1}{2^\ast} \int_\Omega |w|^{2^\ast} dx = \half \left(a - \frac{\lambda}{\lambda_{k+1}}\right) \int_\Omega |\nabla w|^2\, dx + \o(\norm{w}^2) \quad \text{as } \norm{w} \to 0
\end{multline*}
by \eqref{3.30}, so $J(w) \ge 0$ if $\lambda < a \lambda_{k+1}$ and $\norm{w}$ is sufficiently small. So $J$ has at least two nontrivial critical points by Proposition \ref{Proposition 3.12} in the case $(a)$. In the case $(b)$, \eqref{1.7} gives
\[
J(w) \ge \int_\Omega \left[\frac{a}{2}\, \Big(|\nabla w|^2 - \lambda_{k+1}\, w^2\Big) - G(x,w)\right] dx \quad \forall w \in W.
\]
The local sign condition on $G$ implies that the right-hand side is nonnegative when $\norm{w}$ is small (see Li and Willem \cite{MR1312028}). So $J$ has two nontrivial critical points in this case also.
\end{proof}

\def\cdprime{$''$}

\end{document}